\documentclass{article}
\usepackage{amsmath}
\usepackage{amsthm}
\usepackage{amsfonts}
\usepackage{latexsym}
\usepackage{amssymb}
\usepackage{enumerate, url}
\newtheorem{thrm}{Theorem}
\newtheorem{lem}[thrm]{Lemma}
\newtheorem{cor}[thrm]{Corollary}

\newtheorem{claim}[thrm]{Claim}

\newcommand{\NS}{\mathrm{NS}}
\newcommand{\pmax}{\mathbb{P}_{\mathrm{max}}}
\newcommand{\bbP}{\mathbb{P}}
\newcommand{\breals}{\omega^{\omega}}
\newcommand{\AD}{\mathsf{AD}}

\newcommand{\ZFC}{\mathsf{ZFC}}
\newcommand{\MA}{\mathsf{MA}}
\newcommand{\cP}{\mathcal{P}}
\newcommand{\bbH}{\mathbb{H}}

\newcommand{\bbR}{\mathbb{R}}
\newcommand{\cof}{\mathrm{cof}}

\newtheoremstyle{hdefinition}%
  {\topsep}%
  {\topsep}%
  {\upshape}
  {}%
  {\bfseries}%
  {.}
  { }%
  {\thmnumber{#2 }\thmname{#1}\thmnote{ \rm(#3)}}%

\newtheoremstyle{hclaim}%
  {\topsep}%
  {\topsep}%
  {\itshape}%
  {}%
  {\bfseries}%
  {.}
  { }%
  {\thmname{#1}\thmnote{ \rm#3}}%

\newtheoremstyle{hnotation}%
  {\topsep}%
  {\topsep}%
  {\upshape}%
  {}%
  {\bfseries}%
  {.}
  { }%
  {\thmname{#1}\thmnote{ \rm#3}}%

\theoremstyle{hclaim}
\newtheorem*{claim*}{Claim}

\theoremstyle{hdefinition}
\newtheorem{df}[thrm]{Definition}

\newtheorem{remark}[thrm]{Remark}
\newtheorem{ques}[thrm]{Question}

\begin{document}

\title{Strongly increasing sequences}

\author{Chris Lambie-Hanson\thanks{The first author was partially supported by GA\v{C}R project 23-04683S and the Academy of Sciences of the Czech Republic (RVO 67985840).} \and Paul B. Larson\thanks{The research of the second author was partially supported by NSF grants DMS-1201494, DMS-1764320 and DMS-2452139. The main results in this paper were proved in 2015.}}

\maketitle

\begin{abstract}
Using a variation of Woodin's $\pmax$ forcing, we force over a model of the Axiom of Determinacy to produce a model of ZFC containing a very strongly increasing sequence of length $\omega_{2}$ consisting of functions from $\omega$ to $\omega$. We also show that there can be no such sequence of length $\omega_{4}$. 
\end{abstract}

\section{Introduction}

Given functions $f, g$ from $\omega$ to the ordinals, a relation $R$ in $\{ <, >, \leq, \geq, =\}$ and $n \in \omega$, we write
\begin{itemize}
\item $f R _{n} g$ to mean that $f(m) R g(m)$ for all $m \in \omega \setminus n$;
\item $f R^{*} g$ to mean that $\{ n \in \omega : \neg( f(n) R g(n))\}$ is finite.
\end{itemize}

This paper concerns wellordered sequences of functions from the integers to the ordinals, which are increasing in the following senses. 

\begin{df}\label{sidef} Given a ordinals $\eta$ and $\gamma$, we say that a sequence $\langle f_{\alpha} : \alpha < \gamma \rangle$ of functions from $\omega$ to $\eta$ is \emph{strongly increasing} if
\begin{enumerate}
\item for all $\alpha < \beta < \gamma$, $f_\alpha <^* f_\beta$
\item\label{sitwo} for each limit ordinal $\beta < \gamma$, there exist a club $C_\beta \subseteq \beta$
and an $n_\beta < \omega$ such that, for all $\alpha$ in $C_\beta$,
$f_\alpha <_{n_\beta} f_\beta$.
\end{enumerate}
\end{df}

\begin{df}\label{vsidef} We say that a strongly increasing sequence $\langle f_{\alpha} : \alpha < \gamma \rangle$ is \emph{very strongly increasing} if
\begin{enumerate}
\item $n_{\beta}$ in part (\ref{sitwo}) of Definition \ref{sidef} can be chosen to be $0$ for each limit ordinal $\beta$ of uncountable cofinality;
\item\label{vsitwo} for each ordinal of the form $\beta + \omega < \gamma$, and each $n \in \omega$, there is an $\alpha$ in the interval
$[\beta + 1, \beta + \omega)$ such that $f_{\alpha} \leq_{n} f_{\beta + \omega}$ and $f_{\alpha}(m) = 0$ for all $m < n$.
\end{enumerate}
\end{df}

The ordinal $\gamma$ is called the \emph{length} of $\langle f_{\alpha} : \alpha < \gamma \rangle$, and a sequence is said to have \emph{successor length} if its length is a successor ordinal.

\begin{remark}\label{ctworem} Condition (\ref{vsitwo}) of Definition \ref{vsidef} is designed to make condition (\ref{sitwo}) of Definition \ref{sidef} hold automatically when $\beta$ has countable cofinality. In addition, it ensures that 
whenever $\bar{f} = \langle f_{\alpha} : \alpha < \gamma \rangle$ is very strongly increasing and $\gamma$ has countable cofinality and is not of the form $\beta + \omega$, there exists a function $f_{\gamma}$ such that $\langle f_{\alpha} : \alpha \leq \gamma\rangle$ is very strongly increasing;  moreover, $f_{\gamma}$ can be any function $f$ such that $f_{\alpha} <^{*} f$ for all $\alpha < \gamma$. Extending a very strongly increasing sequence $\langle f_{\alpha} : \alpha \leq \gamma\rangle$ to one of the form $\langle f_{\alpha} : \alpha \leq \gamma + \omega\rangle$ can be done by letting $f_{\gamma + \omega}(n)$ be $f_{\gamma}(n) + n$ for each $n \in \omega$, and letting $f_{\gamma + k}(n)$(for all $k,n \in \omega$ with $k$ positive) be $0$ if $n\leq k$ and $f_{\gamma}(n) + k$ otherwise.
\end{remark}

\begin{remark}\label{ctworem2}
If $\langle f_{\alpha} :\alpha < \gamma \rangle$ is a 
strongly increasing sequence, then there is a sequence $\langle f'_{\alpha} :\alpha < \gamma \rangle$ satisfying condition (\ref{vsitwo}) of Definition \ref{vsidef} such that $f_{\alpha} =^{*} f'_{\alpha}$ for all $\alpha  < \gamma$. 
\end{remark} 



Our interest in strongly increasing sequences is partially motivated by two longstanding open 
problems:
\begin{enumerate}
  \item Is the Chang's Conjecture variant $(\aleph_{\omega+1}, \aleph_\omega) \twoheadrightarrow 
  (\aleph_2, \aleph_1)$ consistent?
  \item Is there consistently a pair of inner models $V \subseteq W$ of $\mathsf{ZFC}$ such 
  that $(\aleph_{\omega+1})^V = (\aleph_2)^W$.
\end{enumerate}

In Section~\ref{motivation_section}, we review some of what is known about these questions and 
show that a positive answer to either would entail the consistency of the existence of 
an ordinal $\eta < \omega_2$ and a strongly increasing 
sequence of length $\omega_2$ consisting of functions from $\omega$ to $\eta$. We also show that 
there can exist no such sequence of length $\omega_n$ for any $4 \leq n < \omega$. 

In Section~\ref{pmax_section}, we prove the main result of our paper (Theorem \ref{mainthrmdetail}), using a variant of 
Woodin's $\pmax$ forcing to establish the 
consistency of the existence of a strongly increasing sequence of length $\omega_2$ consisting 
of functions from $\omega$ to $\omega$. The following is a simplified version of the theorem. 

\begin{thrm}
Suppose that $V =L(A, \mathbb{R})$, for some $A \subseteq \breals$, and that $\AD^{+}$ holds. Then there is a forcing extension in which ZFC holds and there exists a very strongly increasing $\omega_{2}$-sequence of functions from $\omega$ to $\omega$. 
\end{thrm}

The axiom $\AD^{+}$ is a strengthening of the Axiom of Determinacy due to Woodin (see \cite{PLextensions}). The theorem implies the weaker version where $A= \emptyset$ and $\AD^{+}$ is replaced with $\AD$. 
The partial order used in the proof is a variation of Woodin's $\pmax$ forcing \cite{Woodin}. The use of the hypothesis AD + $V = L(\mathbb{R})$ and the fact that ZFC holds in the corresponding extension are both part of the standard $\pmax$ machinery. 



\section{Chang's Conjecture and strongly increasing sequences} \label{motivation_section}

In this section we motivate our result by linking it to some well-known
questions about variants of Chang's Conjecture and collapsing successors of singular
cardinals. None of the results in this section is essentially new,\footnote{See, e.g., \cite{Shelah}, \cite{Abraham_Magidor}, or \cite{Sharon_Viale} for related results in slightly different contexts.} but we include them for completeness. Some of the results have not been stated before in the precise forms we need.

We begin by noting the following consequence of Remark \ref{ctworem}, which can be proved by induction on $\beta$. 

\begin{lem}\label{extendlem} If $F$ is a very strongly increasing sequence in $\breals$, $\gamma$ and $\beta$ are ordinals with $\beta < \omega_{2}$ and $\gamma + 1$ is the length of $F$, then there is a very strongly increasing sequence in $\breals$ of length $\gamma + \beta$ extending $F$.
\end{lem}


In particular, there is always a strongly increasing sequence in $\breals$ of length $\omega_1$.
We next show that instances of Chang's Conjecture at singular cardinals of countable cofinality entail the existence of
more interesting strongly increasing sequences. Recall that the Chang's Conjecture principle 
\[(\kappa, \lambda) \twoheadrightarrow (\delta, \gamma)\] says that whenever $M$ is a structure of cardinality $\kappa$ over some countable language, and $B$ is a subset of $M$ of cardinality $\lambda$, then $M$ has an elementary substructure $X$ of cardinality $\delta$ such that $|X \cap B| = \gamma$. We use equivalent formulations of these principles below. 

\begin{lem} \label{cc_lem}
  Suppose that $\mu$ is a singular cardinal of countable cofinality,
  $\kappa < \mu$ is a cardinal, and $(\mu^+, \mu) \twoheadrightarrow (\kappa^+, \kappa)$
  holds. Then there exist $\eta < \kappa^+$ and a strongly increasing sequence
  $\langle f_\alpha : \alpha < \kappa^+ \rangle$ of functions from $\omega$ to $\eta$.
\end{lem}

\begin{proof}
  Let $\langle \mu_i : i < \omega \rangle$ be an increasing sequence of regular cardinals,
  cofinal in $\mu$. 
It is straightforward to recursively construct a strongly  increasing sequence $\vec{g} = \langle g_\beta : \beta < \mu^+ \rangle$ of functions
  in $\prod_{i < \omega} \mu_i$, letting $n_{\beta}$ be the least $i$ such that $\mu_{i} > \cof(\beta)$, for each limit ordinal $\beta$. Let $\theta$ be a sufficiently large regular cardinal,
  and use $(\mu^+, \mu) \twoheadrightarrow (\kappa^+, \kappa)$ to find an elementary
  substructure $N \prec (H(\theta), \in, \mu^+, \mu, \vec{g})$ such that
  $|N \cap \mu^+| = \kappa^+$ and $|N \cap \mu| = \kappa$. Since $N$ contains a surjection from $\mu$ to each element of $\mu^{+} \cap N$, 
$\mathrm{otp}(N \cap \mu^+) = \kappa^+$.

  By elementarity and the fact that $\vec{g}$ is strongly increasing,
  we know that, for every limit ordinal $\delta \in N \cap \mu^+$,
  there exist a club $C_\delta \subseteq \delta$ and an $n_\delta < \omega$ such
  that $C_\delta \in N$ and, for all $\beta \in C_\delta$, we have $g_\beta <_{n_\delta} g_\delta$.
  Let $M$ be the transitive collapse of $N$, and let $\pi:N \rightarrow M$ be the collapse
  map.

  \begin{claim} \label{closed_claim}
    Suppose that $\delta \in N \cap \mu^+$ is a limit ordinal. Then $\pi(C_\delta)$
    is closed in its supremum.
  \end{claim}

  \begin{proof}
    Let $\bar{\beta} < \pi(\delta)$ be such that $\sup(\pi(C_\delta) \cap \bar{\beta}) = \bar{\beta}$.
    Let $\beta = \pi^{-1}(\bar{\beta})$, and note that $\beta$ is a limit ordinal.
    If $\alpha < \beta$, then $\pi(C_\delta) \cap (\pi(\alpha), \bar{\beta}) \neq \emptyset$,
    so $C_\delta \cap (\alpha, \beta) \neq \emptyset$. It follows that
    $N \models ``C_\delta \cap \beta \mbox{ is unbounded in } \beta,"$ so, by elementarity,
    $C_\delta$ is in fact unbounded in $\beta$ and, since $C_\delta$ is closed in $\delta$,
    we have $\beta \in C_\delta$. But then $\bar{\beta} \in \pi(C_\delta)$, as desired.
  \end{proof}

Since $\omega \subseteq N$, it follows that for all $\beta \in N \cap \mu^+$, we have
  $\mathrm{range}(g_\beta) \subseteq N$. Let $\langle \beta_\alpha : \alpha < \kappa^+ \rangle$
  enumerate $M \cap \mu^+$ in increasing order. Define a sequence $\vec{f} = \langle f_\alpha :
  \alpha < \kappa^+ \rangle$ of functions from $\omega$ to $\pi(\mu)$ by letting
  $f_\alpha(n) = \pi(g_{\beta_\alpha}(n))$ for all $\alpha < \kappa^+$ and all $n < \omega$.
Claim~\ref{closed_claim} and the fact that $\pi(\mu) < \kappa^{+}$ imply that $\vec{f}$ is strongly
  increasing, as witnessed by $n_{\beta_\alpha}$ and $\pi(C_{\beta_\alpha})$
  for each limit ordinal $\alpha < \kappa^+$.
\end{proof}

In \cite{Levinski_Magidor_Shelah}, Levinski, Magidor, and Shelah prove,
assuming the consistency of a certain large cardinal hypothesis, the consistency
of a number of Chang's Conjecture variants involving singular cardinals, most
notably $(\aleph_{\omega + 1}, \aleph_\omega) \twoheadrightarrow (\aleph_1, \aleph_0)$.
However, in any instance $(\mu^+, \mu) \twoheadrightarrow (\kappa^+, \kappa)$
known to be consistent in which $\mu$ is a singular cardinal, we have $\mathrm{cf}(\mu) =
\mathrm{cf}(\kappa)$. This leads to the following natural folklore question.

\begin{ques} \label{cc_ques}
  Is $(\aleph_{\omega + 1}, \aleph_\omega) \twoheadrightarrow (\aleph_2, \aleph_1)$
  consistent?
\end{ques}

A closely related question, raised by Bukovsk\'y and Copl\'akov\'a-Hartov\'a in
\cite{Bukovsky}, is the following.

\begin{ques} \label{collapsing_ques}
  Is there consistently a pair of inner models $V \subseteq W$ of $\ZFC$ such that
  $(\aleph_{\omega + 1})^V = (\aleph_2)^W$?
\end{ques}

To see one connection between these two questions, observe that, if 
\[(\aleph_{\omega + 1}, \aleph_\omega) \twoheadrightarrow (\aleph_2, \aleph_1)\] and there is a Woodin cardinal,
then there is a $V$-generic filter $G$ for Woodin's stationary tower forcing such
that $(\aleph_{\omega + 1})^V = (\aleph_2)^{V[G]}$ (see \cite{PLstationary}).

Essentially the same argument as that given in the proof of Lemma~\ref{cc_lem} yields
the following result, which states that a positive answer to Question~\ref{collapsing_ques}
would entail the existence of interesting strongly increasing sequences.

\begin{lem}
  Suppose that $V \subseteq W$ are inner models of $\ZFC$, $n < \omega$, and
  $(\aleph_{\omega + 1})^V = (\aleph_{n+1})^W$. Then, in $W$, there is
  $\eta < \omega_{n+1}$ and a strongly increasing sequence $\langle f_\alpha :
  \alpha < \omega_{n+1} \rangle$ of functions from $\omega$ to $\eta$.
\end{lem}

\begin{proof}[Proof sketch]
  As in the proof of Lemma~\ref{cc_lem}, in $V$ we can construct a strongly increasing sequence 
  $\vec{f} = \langle f_\alpha : \alpha < \omega_{\omega+1} \rangle$ in $\prod_{i < \omega} 
  \omega_i$. In $W$, we have $(\omega_{\omega+1})^V = \omega_{n+1}$, and $\vec{f}$ remains a 
  strongly increasing sequence of functions. Thus, letting $\eta = (\omega_\omega)^V < 
  (\omega_{n+1})^W$, we see that $\vec{f}$ is as desired.
\end{proof}

For more information and partial progress on Questions \ref{cc_ques} and \ref{collapsing_ques},
we refer the reader to \cite{Sharon_Viale} and \cite{Cummings}.

For the next result, we need some notions from PCF theory \cite{Shelah}. The following is a special
case of a more general definition.

\begin{df}
  Suppose that $\vec{f} = \langle f_\alpha : \alpha < \gamma \rangle$ is a $<^*$-increasing
  sequence of functions from $\omega$ to the ordinals. A function $g$ from $\omega$ to the
  ordinals is an \emph{exact upper bound} (eub) for $\vec{f}$ if:
  \begin{itemize}
    \item $g$ is an \emph{upper bound}, i.e., for every $\alpha < \gamma$, we have $f_\alpha <^* g$;
    \item for every function $h$ such that $h <^* g$, there is $\alpha < \gamma$ such that
      $h <^* f_\alpha$.
  \end{itemize}
  An ordinal $\beta < \gamma$ is called \emph{good} for $\vec{f}$ if
  $\mathrm{cf}(\beta) > \omega$ and there is an eub
  $h$ for $\vec{f} \restriction \beta = \langle f_\alpha : \alpha < \beta \rangle$
  such that $\mathrm{cf}(h(i)) = \mathrm{cf}(\beta)$ for all but finitely many $i < \omega$.
\end{df}

 Recall that $S^{\lambda}_{\kappa}$ denotes the set of ordinals below $\lambda$ of cofinality $\kappa$. The following theorem, due in a much more general form to Shelah \cite{Shelah}, is a basic result in PCF theory. For a proof, we refer the reader to \cite[Theorem 10.1]{Cummings_notes}.

\begin{thrm} \label{eub_thrm}
  Suppose that $\kappa < \lambda$ are uncountable regular cardinals and that
  $\vec{f} = \langle f_\alpha : \alpha < \lambda \rangle$ is a $<^*$-increasing
  sequence of functions from $\omega$ to the ordinals. If there are stationarily
  many $\beta \in S^\lambda_\kappa$ such that $\beta$ is good for $\vec{f}$,
  then $\vec{f}$ has an eub $h$ such that $\mathrm{cf}(h(i)) > \kappa$ for all
  $i < \omega$.
\end{thrm}

We can now prove the following result bounding the lengths of strongly increasing sequences of functions. 

\begin{thrm}\label{goodthrm}
  Suppose that $\lambda$ is an uncountable regular cardinal and $\epsilon < \lambda^{+3}$.
  Then there is no strongly increasing sequence $\vec{f} = \langle f_\alpha : \alpha < \lambda^{+3} \rangle$
  of functions from $\omega$ to $\epsilon$.
\end{thrm}

\begin{proof}
  Suppose for the sake of contradiction that $\vec{f} = \langle f_\alpha :
  \alpha < \lambda^{+3} \rangle$ is a strongly increasing sequence of functions
  from $\omega$ to $\epsilon$. Fix a club-guessing sequence $\langle C_\xi :
  \xi \in S^{\lambda^{++}}_\lambda \rangle$, i.e., a sequence such that each $C_\xi$ is
  a club in $\xi$ of order type $\lambda$ and, for every club $D \subseteq \lambda^{++}$,
  there is a $\xi \in S^{\lambda^{++}}_\lambda$ such that $C_\xi \subseteq D$ (see Claim 
  2.14B of section III(2) of \cite{Shelah}, for 
  instance, for a proof that this can be done).

  Our goal is to show that there are stationarily many elements of
  $S^{\lambda^{+3}}_\lambda$ that are good for $\vec{f}$.
  To this end, fix a club $C \subseteq \lambda^{+3}$.
  We now construct an increasing, continuous sequence $\langle \alpha_\eta :
  \eta < \lambda^{++} \rangle$ of elements of $C$.
  Begin by letting $\alpha_0 = \min(C)$. If $\eta < \lambda^{++}$ is a limit ordinal,
  then $\alpha_\eta = \sup\{\alpha_\zeta : \zeta < \eta\}$. Suppose
  now that $\eta < \lambda^{++}$ and $\langle \alpha_\zeta : \zeta \leq \eta \rangle$
  has been constructed. We show how to obtain $\alpha_{\eta + 1}$.

  For each $\xi \in S^{\lambda^{++}}_\lambda$, ask whether there exist
  $\beta < \lambda^{+3}$ and $n < \omega$ such that, for all $\zeta \in C_\xi \cap
  (\eta + 1)$, we have $f_{\alpha_\zeta} <_n f_\beta$. Note that, if $\beta$
  has this property, then so does every $\gamma$ with $\beta \leq \gamma < \lambda^{+3}$.
  If the answer is yes, then choose an ordinal $\alpha^\xi_\eta \in C \setminus (\alpha_\eta + 1)$
  witnessing this. If the answer is no, then let $\alpha^\xi_\eta = \min(C \setminus (\alpha_\eta + 1))$.
  Let $\alpha_{\eta + 1} = \sup\{\alpha^\xi_\eta : \xi \in S^{\lambda^{++}}_\lambda\}$.

  Let $\beta = \sup\{\alpha_\eta : \eta < \lambda^{++}\}$. By the fact that $\vec{f}$
  is strongly increasing, we can find a club $D \subseteq \beta$ and a natural
  number $n$ such that, for all $\alpha \in D$, we have $f_\alpha <_n f_\beta$.
  Let $E = \{\eta < \lambda^{++} : \alpha_\eta \in D\}$, and note that $E$ is club
  in $\lambda^{++}$. Fix $\xi \in S^{\lambda^{++}}_\lambda$ such that
  $C_\xi \subseteq E$. For $\eta \in C_\xi$, let $\eta^{\dagger}$ denote
  $\min(C_\xi \setminus (\eta + 1))$.

  Suppose $\eta \in C_\xi$. When defining $\alpha_{\eta + 1}$, the answer to the question
  asked about $C_\xi$ was ``yes,'' as witnessed by $\beta$. Therefore, $\alpha_{\eta + 1}$
  was chosen to be large enough so that, for some $n < \omega$ and all $\zeta \in C_\xi
  \cap (\eta + 1)$, we have $f_{\alpha_\zeta} <_n f_{\alpha_{\eta + 1}}$. The same obviously
  holds for all $\eta'$ in the interval $(\eta, \lambda^{++})$. In particular, there
  is a natural number $n_\eta$ such that, for all $\zeta \in C_\xi \cap (\eta + 1)$,
  we have $f_{\alpha_\zeta} <_{n_\eta} f_{\alpha_{\eta^\dagger}}$.

  Since $\lambda$ is regular and uncountable, we can fix a natural number $n$ and an unbounded
  set $A \subseteq C_\xi$ such that, for all $\eta \in A$, we have $n_\eta = n$.
  Let $B = \{\alpha_{\eta^\dagger} : \eta \in A\}$. Then $B$ is unbounded in
  $\alpha_\xi$ and, for all $\zeta < \eta$, both in $A$, we have $f_{\alpha_{\zeta^\dagger}} <_n f_{\alpha_{\eta^\dagger}}$.
  But now, if $g$ is a function from $\omega$ to the ordinals such that, for all $n \leq m < \omega$,
  we have $g(m) = \sup\{f_{\alpha_{\eta^\dagger}}(m) : \eta \in A\}$, it is easily verified that $g$
  is an eub for $\vec{f} \restriction \alpha_\xi$. But then $g$ witnesses that $\alpha_\xi$
  is good for $\vec{f}$. Moreover, by construction, $\alpha_\xi \in C$. Since $C$ was arbitrary,
  we have shown that there are stationarily many elements of $S^{\lambda^{+3}}_\lambda$
  that are good for $\vec{f}$.

  By Theorem~\ref{eub_thrm}, it follows that there is an eub $h$ for $\vec{f}$
  such that $\mathrm{cf}(h(i)) > \lambda$ for all $i < \omega$.

  \begin{claim}
    $\mathrm{cf}(h(i)) \geq \lambda^{+3}$ for all but finitely many $i < \omega$.
  \end{claim}

  \begin{proof}
    If not, then there exist $k \in \{1,2\}$ and an unbounded $A \subseteq \omega$
    such that, for all $i \in A$, we have $\mathrm{cf}(h(i)) = \lambda^{+k}$.
    For each $i \in A$, let $\{\delta^i_\eta : \eta < \lambda^{+k}\}$
    enumerate, in increasing fashion, a set of ordinals cofinal in $h(i)$.
    For each $\eta < \lambda^{+k}$, define a function $h_\eta$ from $\omega$
    to the ordinals by letting $h_\eta(i) = \delta^i_\eta$ if $i \in A$
    and $h_\eta(i) = 0$ otherwise. For each $\eta < \lambda^{+k}$,
    we have $h_\eta <^* h$, so, since $h$ is an eub for $\vec{f}$,
    there is $\beta_\eta < \lambda^{+3}$ such that $h_\eta <^* f_{\beta_\eta}$.
    Let $\gamma = \sup\{\beta_\eta : \eta < \lambda^{+k}\}$. Since $k < 3$,
    we have $\gamma < \lambda^{+3}$. Therefore, for all $\eta < \lambda^{+k}$,
    we have $h_\eta <^* f_\gamma$. Fix an unbounded $B \subseteq \lambda^{+k}$
    and an $n < \omega$ such that, for all $\eta \in B$, we have $h_\eta
    <_n f_\gamma$. But then, for all $i \in A \setminus n$,
    we must have $f_\gamma(i) \geq \sup\{\delta^i_\eta : \eta \in B\} = h(i)$,
    contradicting the fact that $h$ is an upper bound for $\vec{f}$.
  \end{proof}

  But this claim immediately contradicts the fact that $\vec{f}$ is a sequence of
  functions from $\omega$ to $\epsilon$ and $\epsilon < \lambda^{+3}$. This is
  because, by the claim, we must have $h(i) > \epsilon$ for all but finitely many
  $i < \omega$. But then the constant function, taking value $\epsilon$, witnesses
  that $h$ fails to be an eub.
\end{proof}

The results in this section lead to the following corollary. We note that clauses (2) and (3) 
of the corollary were already known via slightly different proofs; see, e.g., \cite{Cummings} 
and \cite{Sharon_Viale}.

\begin{cor}\label{no4cor}
  Suppose that $3 \leq n < \omega$.
  \begin{enumerate}
    \item If $\eta < \omega_{n+1}$, then there is no strongly increasing sequence
      $\langle f_\alpha : \alpha < \omega_{n+1} \rangle$ of functions from
      $\omega$ to $\eta$.
    \item $(\aleph_{\omega + 1}, \aleph_\omega) \not\twoheadrightarrow (\aleph_{n+1},
      \aleph_n)$.
    \item There are no inner models $V \subseteq W$ of $\ZFC$ such that $(\aleph_{\omega + 1})^V
      = (\aleph_{n+1})^W$.
  \end{enumerate}
\end{cor}

It also follows that the only regular cardinals that can possibly be lengths of
strongly increasing sequences from $\breals$ are $\aleph_n$ for $0 \leq n \leq 3$.
We have seen that there are always such sequences of length $\aleph_0$ and $\aleph_1$.
We will prove, in Section~\ref{pmax_section}, the consistency of the existence of a
strongly increasing sequence of length $\aleph_2$. The question about the consistency
of the existence of a strongly increasing sequence of length $\aleph_3$ remains open.




\section{Consistency via a $\pmax$ variation} \label{pmax_section}

In this section we use a natural variation of Woodin's partial order $\pmax$ to force over a model of $\AD^{+} + \exists A \subseteq \bbR\, V =L(A,\bbR)$ to produce a model of $\ZFC$ with a very strongly increasing sequence in $\breals$ of length $\omega_{2}$. We refer the reader to \cite{PLextensions} for background on $\AD^{+}$. Since $L(\bbR) \models \AD \rightarrow \AD^{+}$ (by a result of Kechris), our result applies to models of $\AD$ of the form $L(\bbR)$. 
We refer the reader to \cite{Woodin} for background on $\pmax$, especially Chapter 4 and Section 9.2.
The article \cite{PLhandbook} may also be helpful. 




Conditions in our partial order $\bbP$ are triples $(M, F, a)$ such that
\begin{itemize}
\item $M$ is a countable transitive model of $\ZFC + \MA_{\aleph_{1}}$ which is iterable with respect to $\NS_{\omega_{1}}^{M}$;
\item $F$ is in $M$ a very strongly increasing sequence of successor length;
\item $a$ is an element of $\cP(\omega_{1})^{M}$ such that, for some $x \in \cP(\omega)^{M}$, $\omega_{1}^{M} = \omega_{1}^{L[a, x]}$.
\end{itemize}
The order is : $(M, F, a) < (N, H, b)$ if there exists in $M$ an iteration \[j \colon (N, \NS_{\omega_{1}}^{N}) \to (N^{*}, \NS_{\omega_{1}}^{N^{*}})\] such that
\begin{itemize}
\item $j(b) = a$;
\item $\NS_{\omega_{1}}^{N^{*}} = \NS_{\omega_{1}}^{M} \cap N^{*}$;
\item $j(H)$ is a proper initial segment of $F$.
\end{itemize}

The iterations referred to in the definition above are repeated generic elementary embeddings induced by forcing with $\cP(\omega_{1})/\NS_{\omega_{1}}$; a model is iterable if this process always results in wellfounded models (see Definition 3.5 of \cite{Woodin}). The map $j$ in the definition of the order embeds $N$ elementarily into $N^{*}$; $N^{*}$ need not be an elementary substructure of $M$. However, since $N^{*}$ and $M$ have the same $\omega_{1}$, and agree about stationarity for subsets of $\omega_{1}^{M}$ in $N^{*}$, $j(H)$ is a very strongly increasing sequence in $M$. 
The requirement above that $F$ properly extends $j(H)$ simplifies the arguments below (by removing trivial cases), and adds no new complications, by Lemma \ref{extendlem}. Note that $j$ will be an element of $H(\aleph_{2})^{M}$; moreover, the ordinal height of $N^{*}$ will be less than $\omega_{2}^{M}$. 

\begin{remark}\label{PPmaxrem} If $(M, F, a)$ is a condition in our partial order $\bbP$, then $\langle (M, \NS_{\omega_{1}}^{M}), a \rangle$ is a condition in $\pmax$. Conversely, if $\langle (M, \NS_{\omega_{1}}^{M}), a \rangle$ is a $\pmax$ condition, and $M \models ``F$ is a very strongly increasing sequence of successor length", then $(M, F, a)$ is a condition in $\bbP$. If $(M, F, a)$ and $(N, G, b)$ are conditions in $\bbP$ such that $\langle (M, \NS_{\omega_{1}}^{M}), a \rangle < \langle (N, \NS_{\omega_{1}}^{N}), b \rangle$ (as $\pmax$ conditions, as witnessed by an embedding $j$), then there is an $F'$ in $M$ such that $(M, F', a) < (N, G, b)$ (as conditions in $\bbP$); for instance $F'$ can be an extension of $j(G)$ formed by adding one additional member. Moreover, by Lemma \ref{extendlem}, $F'$ can be chosen in $M$ to have any desired length below $\omega_{2}^{M}$. Finally, the requirement that $M \models \MA_{\aleph_{1}}$ and the conditions on the set $a \in \cP(\omega_{1})^{M}$ imply that the order on each comparable pair of $\bbP$-conditions is witnessed by a unique iteration (this follows immediately from the corresponding fact for $\pmax$). 
\end{remark}

\begin{remark}\label{tooeasy} Much of the standard $\pmax$ machinery can be applied directly to the partial order $\bbP$. In particular, $\AD^{+}$ implies the following facts about $\bbP$, each of which can be derived from
Theorem 9.31 of \cite{Woodin} and Remark \ref{PPmaxrem} (the first follows from the second). The assertion that $(N, \NS_{\omega_{1}}^{N})$ is $A$-iterable means that $A \cap N \in N$, and, whenever $j \colon (N, \NS_{\omega_{1}}^{N}) \to (N^{*}, \NS_{\omega_{1}}^{*})$ is an iteration, $j(A \cap N) = A\cap N^{*}$. 
\begin{itemize}
\item For each set $x \in H(\aleph_{1})$, there exists a $\bbP$-condition $(M, F, a)$ with $x$ in $H(\aleph_{1})^{M}$.
\item For each $\bbP$ condition $(M, F, a)$ and each $A \subseteq \breals$ there is a $\bbP$-condition
$(N, F', b) < (M, F, a)$ such that
\begin{itemize}
\item $(N, \NS_{\omega_{1}}^{N})$ is $A$-iterable;
\item $\langle H(\aleph_{1})^{N}, \in , A \cap N \rangle \prec \langle H(\aleph_{1}), \in, A \rangle$;
\end{itemize}
\end{itemize}
\end{remark}

In conjunction with the previous two remarks, the proof of Theorem 4.43 of \cite{Woodin} gives the $\omega$-clousure of $\bbP$. We sketch the proof since the same situation appears again in the proof of Theorem \ref{newlem}. 

\begin{lem}\label{cclem} Every descending $\omega$-sequence in $\bbP$ has a lower bound.
\end{lem} 

\begin{proof} (Sketch)
Suppose that $p_{i} = (M_{i}, F_{i}, a_{i})$ ($i \in \omega$) are the members of a descending $\omega$-sequence in $\bbP$. By the first part of Remark \ref{tooeasy} we can work inside a countable transitive model $N$ such that that $\langle p_{i} : i \in \omega \rangle \in H(\aleph_{1})^{N}$ and $(N, \NS_{\omega_{1}}^{N})$ is iterable. As in the proof of Theorem 4.43 of \cite{Woodin}, by combining the iterations witnessing the order on the $p_{i}$'s we get a sequence of $\bbP$-conditions 
$(\hat{M}_{i}, \hat{F}_{i}, \hat{a}_{i})$ $(i \in \omega)$ such that, for each $i \in \omega$,  
\begin{itemize}
\item $(\hat{M}_{i}, \NS_{\omega_{1}}^{\hat{M}_{i}})$ is an iterate of $(M_{i}, \NS_{\omega_{1}}^{M_{i}})$ and $\hat{F}_{i}$ and $\hat{a}_{i}$ are the corresponding images of $F_{i}$ and $a_{i}$ respectively, 
\item $\omega_{1}^{\hat{M}_{i}} = \omega_{1}^{\hat{M}_{0}}$, 
\item $\NS_{\omega_{1}}^{\hat{M}_{i}} = \NS_{\omega_{1}}^{\hat{M}_{i+1}} \cap \hat{M}_{i}$, 
\item $\hat{F}_{i}$ is a proper initial segment of $\hat{F}_{i+1}$ and 
\item $\hat{a}_{i} = \hat{a}_{0}$. 
\end{itemize} 
One can then iterate the sequence $\langle \hat{M}_{i} : i \in \omega \rangle$ (as in Defnition 4.15 of \cite{Woodin}) in $N$ to produce a sequence of iterations
\[j_{i} \colon (\hat{M}_{i}, \NS^{\hat{M}_{i}}_{\omega_{1}}) \to (M^{*}_{i}, \NS^{M^{*}_{i}}_{\omega_{1}})\hspace{.1in}(i \in \omega)\] such that 
\begin{itemize}
\item each $M^{*}_{i}$ is correct (in $N$) about stationary subsets of $\omega_{1}$;
\item $\bigcup_{i \in \omega}j_{i}(\hat{F}_{i})$ (which we will call $F$) is (in $N$) a very strongly increasing sequence in $\breals$ whose length is a limit ordinal of countable cofinality;  
\item each $j_{i}(\hat{a}_{i})$ is the same set (which we will call $a$). 
\end{itemize} 
There exists by Remark \ref{ctworem} a very strongly increasing sequence $F^{*}$ in $N$ which extends $F$ by the addition of one function. Then $(M, F^{*}, a)$ is a lower bound for $\langle p_{i} : i \in \omega \rangle$.  
\end{proof}



The remaining argument concerns descending $\omega_{1}$-sequences in $\bbP$, from the point of view of some countable transitive model. The main point of this argument is to prove the second item from Remark \ref{summaryrem} below: in the $\bbP$-extension, every subset of $\omega_{1}$ arises from an iteration of a model coming from a condition in the generic filter (the model that we are working in can be taken to satisfy the second part of Remark \ref{tooeasy} with respect to some set of reals $A$ coding a $\bbP$-name for a subset of $\omega_{1}$). Less importantly for us, the same construction is also used to show that the nonstationary ideal is saturated in the $\bbP$-extension.

As in Lemma 5.2 of \cite{PLhandbook}, we phrase the construction of our descending $\omega_{1}$-sequence in terms of a game, where player $I$ takes care of certain steps (typically, meeting each member of some $\aleph_{1}$-sized collection of dense sets) which are repeated without change in our context. 

Given a condition $(M, F, a)$, the $\omega_{1}$-\emph{sequence game} for $p = (M, F, a)$ is the game of length $\omega_{1}$ where players $I$ and $II$ pick the members of a descending $\omega_{1}$-sequence of conditions $p_{\alpha} = (M_{\alpha}, F_{\alpha}, a_{\alpha})$ $(\alpha < \omega_{1})$ from $\bbP$, where $p_{0} = p$, player $I$ picks $p_{\alpha}$ for all successor ordinals $\alpha$ and player $II$ picks $p_{\alpha}$ for all limit ordinals $\alpha$.
We have in addition that for each $\alpha < \omega_{1}$, letting $j_{\alpha, \alpha + 1}$ be the embedding of $M_{\alpha}$ into $M_{\alpha +1}$ witnessing that $p_{\alpha + 1} < p_{\alpha}$, $F_{\alpha + 1}$ properly extends $j_{\alpha, \alpha + 1}(F_{\alpha})$.
Such a sequence induces (via composing the embeddings $j_{\alpha, \alpha + 1}$) an elementary embedding $k_{\alpha}$ of each model $M_{\alpha}$ into a structure $M_{\alpha}^{*}$ of cardinality $\aleph_{1}$, and a sequence $F^{*} = \langle f_{\alpha}^{*} : \alpha < \gamma \rangle$ which is the union of the sequences $k_{\alpha}(F_{\alpha})$. Player $II$ wins a run of the game if and only if
\begin{itemize}
\item each $M_{\alpha}^{*}$ is correct about stationary subsets of $\omega_{1}$, and
\item there is a very strongly increasing sequence in $\breals$ of successor length extending $F^{*}$.
\end{itemize}
The last condition above is equivalent to the existence of a function $f \in \breals$ such that $\{ \alpha < \gamma : f_{\alpha} <_{0} f\}$ contains a club subset of $\gamma$ (which will have cofinality $\aleph_{1}$).

Lemma \ref{newlem} below is the only new ingredient in adapting the standard $\pmax$ machinery to give the desired consistency result. The lemma is an adaptation of Lemma 4.46 of \cite{Woodin}; the steps involving the sequences $\sigma$ and the use of Hechler and club-shooting forcing are the only new elements.
For those unfamiliar with $\pmax$, the most mysterious part of the proof may be the iteration of the sequence $\langle \hat{M}_{\beta, i} : i < \omega \rangle$. For details on this sort of construction (which was already sketched in the proof of Lemma \ref{cclem}) we refer the reader to Corollary 4.20 of \cite{Woodin}.
According to the standard arguments, the lemma is applied inside a countable transitive model $M$ which is correct about $\mathbb{P}$ (i.e., such that \[\langle H(\aleph_{1})^{M}, \in, A\rangle \prec \langle H(\aleph_{1}), \in, A\rangle\] for some set of reals $A$ coding $\bbP$, and typically a $\bbP$-name for a subset of $\omega_{1}$) and which contains a Woodin cardinal, as given for instance by Theorem 9.30 of \cite{Woodin}.
The forcing in the proof of Lemma \ref{newlem} (Hechler forcing followed by adding a club subset of $\omega_{1}$) is then followed by a two-stage forcing to produce a condition in $\bbP$ (a lower bound for the $\omega^{M}_{1}$-sequence constructed), first collapsing the Woodin cardinal to be $\omega_{2}$ to make the nonstationary ideal on $\omega_{1}$ precipitous, and then forcing with a c.c.c. forcing (which preserves precipitousness) to make $\MA_{\aleph_{1}}$ hold (see, for instance, the top of page 628 in \cite{Woodin}). 


We will use the following objects from the theory of cardinal characteristics of the continuum. The cardinal characteristic $\mathfrak{d}$ is defined to be the least cardinal $\kappa$ such that there is a set $A \subseteq \breals$ of cardinality $\kappa$ such that for every $f \in \breals$ there is a $g \in A$ such that $f <^{*} g$ (see \cite{ABhandbook}).  We let $(\bbH, \leq_{\bbH})$ denote Hechler forcing \cite{Hechler}, where
\begin{itemize}
\item $\bbH$ is the set of pairs $(s, P)$ such that $s \in \omega^{<\omega}$ and $P \in [\breals]^{<\aleph_{0}}$ and
\item $(s, P) \leq_{\bbH} (s', P')$ if $s$ extends $s'$, $P$ contains $P'$ and
$s(n) > f(n)$ for all $f \in P'$ and $n \in |s| \setminus |s'|$.
\end{itemize}
The partial order $\bbH$ adds an element of $\breals$ which is $<^{*}$-above each element of $\breals$ in the ground model.
Moreover, $\bbH$ is c.c.c. (in fact, $\sigma$-centered) as conditions with the same first coordinate are compatible, so it preserves stationary subsets of $\omega_{1}$. 
It follows that, for any cardinal $\kappa$,  $\MA_{\kappa}$ implies that
$\mathfrak{d} > \kappa$.

\begin{lem}\label{newlem} For each condition $p=(M, F, a)$ in $\bbP$, there is a strategy for player $II$ in the $\omega_{1}$-sequence game for $p$ such that each run of the game in $V$ according to this strategy is winning for player $II$ in some forcing extension which preserves the stationarity of each stationary subset of $\omega_{1}$ in each model $M^{*}_{\alpha}$ arising from the run of the game.
\end{lem}

\begin{proof}
A run of the game builds a descending $\omega_{1}$ sequence of $\bbP$ conditions
$p_{\alpha} = (M_{\alpha}, F_{\alpha}, a_{\alpha})$, with (stationary set preserving, but not elementary) embeddings
$j_{\alpha, \beta} \colon M_{\alpha} \to M_{\beta}$.
Let $\gamma_{\alpha}$ be such that
$F_{\alpha}$ has length $\gamma_{\alpha} + 1$ in $M_{\alpha}$. The definition of the order on $\bbP$ requires that $\gamma_{\beta} > j_{\alpha, \beta}(\gamma_{\alpha})$ for all
$\alpha < \beta$.

We show how to play for player $II$, i.e., how to choose $M_{\beta}$, $F_{\beta}$ and $a_{\beta}$ for each limit ordinal $\beta$, assuming that
$p_{\alpha}$ $(\alpha < \beta)$ have already been chosen.
Let us say that a countable limit ordinal $\beta$ is \emph{relevant} if $\beta$ is the supremum of $\{ \omega_{1}^{M_{\alpha}} : \alpha <\beta\}$. 
At the end of the game, the set of relevant ordinals will be a club subset of $\omega_{1}$. 
For limit ordinals $\beta$ which are not relevant, player $II$ can let $p_{\beta}$ be any lower bound for $\{ p_{\alpha} : \alpha < \beta\}$ (lower bounds exist by Lemma \ref{cclem}). 

As we carry out our construction, by some bookkeeping we associate to each pair
$(\sigma, A)$, where $\sigma \in \omega^{<\omega}$ and $A$ is, in some $M_{\alpha}$ a stationary subset of $\omega_{1}^{M_{\alpha}}$, a stationary
set $B_{\sigma, A} \subseteq \omega^{V}_{1}$, in such a way that the associated sets $B_{\sigma, A}$ are disjoint for distinct pairs $(\sigma, A)$.
We show how to play for $II$ in such a way that, for all relevant limit ordinals $\beta$, if $\beta$ is in $B_{\sigma, A}$, for some $A$ which is a stationary subset of $\omega_{1}^{M_{\alpha'}}$ for some $\alpha' < \beta$, then $\beta$ is in the induced image of $A$ (i.e., $j_{\alpha', \alpha}(A)$) for all $\alpha \in [\beta, \omega_{1}]$ (note that $\omega_{1}^{M_{\beta}}$ will be greater than $\beta$) and $\sigma$ is an initial segment of the first member of $F_{\beta}$ which is not in the images of the preceding $F_{\alpha}$'s (i.e., the member of $F_{\beta}$ indexed by the supremum of $\{j_{\alpha, \beta}(\gamma_{\alpha}) : \alpha < \beta\}$).

For each relevant limit ordinal $\beta$, in round $\beta$ the triples
$(M_{\alpha}, F_{\alpha}, a_{\alpha})$ $(\alpha < \beta)$ will have been chosen, along with the maps $j_{\alpha, \delta}$ ($\alpha \leq \delta < \beta$) witnessing that this is a descending sequence of conditions. Choosing a cofinal $\omega$-sequence $\langle \alpha_{i} : i \in \omega \rangle$ in $\beta$, and composing the maps $j_{\alpha_{i}, \alpha_{i+1}}$ $(i \in \omega)$, we get a
sequence $\langle \hat{M}_{\beta, i} : i < \omega \rangle$, where each $\hat{M}_{\beta, i}$ is the corresponding image of $M_{\alpha_{i}}$.
Let $\hat{\j}_{\alpha_{i}, \beta}$ be the induced embedding of $M_{\alpha_{i}}$ into $\hat{M}_{\beta, i}$, for each $i < \omega$. Then $\omega_{1}^{\hat{M}_{\beta, i}} = \beta$ for all $i < \omega$, and each $\hat{M}_{\beta, i}$ is a subset of the corresponding $\hat{M}_{\beta, i + 1}$, and correct in $\hat{M}_{\beta, i + 1}$ about stationary subsets of its $\omega_{1}$. Fix a countable transitive model $M_{\beta}$ of $\ZFC$ with $\langle \hat{M}_{\beta, i} : i < \omega \rangle$, 
$\langle p_{\alpha} : \alpha < \beta \rangle$ and $\{ \alpha_{i} : i \in \omega \}$ in $H(\aleph_{1})^{M_{\beta}}$, and such that $(M_{\beta}, \NS_{\omega_{1}}^{M_{\beta}})$ is iterable. 

Working in $M_{\beta}$, iterate $\langle \hat{M}_{\beta, i} : i < \omega \rangle$ (building an embedding $j^{*}_{\beta}$ of $\bigcup_{i \in \omega}\hat{M}_{\beta, i}$ into $M_{\beta}$) in such a way that
\begin{itemize}
\item  if $\beta$ is in $B_{\sigma, A}$, for some $\sigma \in \omega^{<\omega}$ and some $A$ which is a stationary subset of $\omega_{1}^{M_{\alpha'}}$ for some $\alpha' < \beta$, then the corresponding image of $A$ ($\hat{\j}_{\alpha_{i}, \beta}(j_{\alpha',\alpha_{i}}(A))$ for the least $i$ such that $\alpha_{i} \geq \alpha'$) is in the first filter in the iteration (this ensures that $\beta \in j^{*}_{\beta}(\hat{\j}_{\alpha_{i}, \beta}(j_{\alpha',\alpha_{i}}(A)))$);
\item each model $j^{*}_{\beta}(\hat{M}_{\beta, i})$ is correct in $M_{\beta}$ about stationary subsets of $\omega_{1}$.
\end{itemize}
Having constructed $j^{*}_{\beta}$, extend the union of the sets $j^{*}_{\beta}(\hat{\j}_{\alpha_{i},\beta}(F_{\alpha_{i}}))$ (the length of which has countable cofinality, since each $\gamma_{\alpha_{i+1}}$ is greater than $j_{\alpha_{i}, \alpha_{i+1}}(\gamma_{\alpha_{i}})$) with one element having $\sigma$ as an initial segment (which can be done by Remark \ref{ctworem}), and let $F_{\beta}$ be this extension. Let $a_{\beta}$ be the common value of $j^{*}_{\beta}(\hat{\j}_{\alpha_{i}, \beta}(a_{\alpha_{i}}))$ $(i < \omega)$. This completes the choice of $M_{\beta}$, $F_{\beta}$ and $a_{\beta}$.

Having constructed the entire run of the game, letting $a$ be the union of the sets $a_{\alpha}$ $(\alpha < \omega_{1})$, there is for each $\alpha < \omega_{1}$ a unique iteration $j_{\alpha, \omega_{1}}$ of $(M_{\alpha}, \NS_{\omega_{1}}^{M_{\alpha}})$ sending $a_{\alpha}$ to $a$. Let $\gamma^{*}_{\alpha}$ denote
the image of $\gamma_{\alpha}$ under this iteration, and let $F^{*}_{\alpha}$ be the corresponding 
image of $F_{\alpha}$. Noting that the sequences $F^{*}_{\alpha}$ extend one another, let $F^{*} = \bigcup_{\alpha < \omega_{1}}F^{*}_{\alpha}$ and let $\gamma^{*}$ be the length of $F^{*}$. 
For each countable limit ordinal $\beta$, let $\eta_{\beta} = \sup\{ \gamma^{*}_{\alpha} : \alpha < \beta\}$.
The set $\{\eta_{\beta} : \beta < \omega_{1} \}$ is closed below its supremum $\gamma$, which has cofinality $\omega_{1}$.

Force (over the model we have been working in, which we call $V$) with Hechler forcing to add a $g \in \breals$ dominating each element of $\omega^{\omega}\cap V$ mod-finite. This forcing is c.c.c., and therefore preserves stationary sets. 

We claim that, in $V[g]$, for each $A$ which is, for some $\alpha < \omega_{1}$, a stationary subset of $\omega_{1}^{M_{\alpha}}$ in $M_{\alpha}$, the set of $\beta \in j_{\alpha, \omega_{1}}(A)$ such that $g$ dominates $f^{*}_{\eta_{\beta}}$ everywhere is stationary. To see that this is the case, consider a Hechler condition
$(\sigma, P)$ and an $A$. In any club $C \subseteq \omega_{1}$, we can find a $\beta \in C \cap B_{\sigma, A}$ such that \[\sup\{\omega_{1}^{M_{\alpha}} : \alpha < \beta\} = \beta.\] Then $(\sigma, P \cup \{ f^{*}_{\eta_{\beta}}\}) \leq (\sigma, P)$. Since (by the construction above) $\sigma$ is an initial segment of $f^{*}_{\eta_{\beta}}$, $(\sigma, P \cup \{\beta\})$ forces that $g(n) \geq f^{*}_{\eta_{\beta}}(n)$ will hold for all $n \in \omega$ (i.e., $g \geq_{0} f^{*}_{\eta_{\beta}}$).

Finally, force to shoot a club $E$ (via the standard forcing with countable conditions) through the set of $\beta < \omega_{1}$ such that $g \geq_{0} f^{*}_{\eta_{\beta}}$. By the previous paragraph, this forcing preserves the stationarity of each set $j_{\alpha, \omega_{1}}(A)$. It follows that in $V[g][E]$, letting $g'(n) = g(n) + 1$ for all $n \in \omega$, $F^{*} \cup \{(\gamma^{*}, g')\}$ witnesses that the run of the $\omega_{1}$-sequence game just produced is winning for player $II$.
\end{proof}

\begin{remark}\label{summaryrem} Lemma \ref{newlem} and standard $\pmax$ arguments (see Theorems 9.32 and 9.34 of \cite{Woodin}) give the following.
\begin{enumerate}
\item If $G \subseteq \bbP$ is a filter, and $A_{G} = \bigcup\{ a : (M, F, a) \in G\}$, then for each $p=(M,F,a) \in G$ there is a unique
iteration $j_{p,G}$ of $(M, \NS_{\omega_{1}}^{M})$ sending $a$ to $A_{G}$.
\item If $W$ is an inner model of $\AD^{+}$ and $G \subseteq \bbP$ is $W$-generic, then every element of $\cP(\omega_{1})^{W[G]}$ is an element of
$j_{p, G}[\cP(\omega_{1})^{M}]$ for some  $p = (M, F, a)$ in $G$.
\item If $A \subseteq \bbR$ is such that $L(A, \bbR) \models \AD^{+}$, then forcing with $\bbP$ over $L(A, \bbR)$ does not add $\omega$-sequences from the ground model, preserves $\omega_{1}$ and $\omega_{2}$ as cardinals, and makes the cardinal $\Theta$ of the ground model the $\omega_{3}$ of the extension. 
\end{enumerate}
\end{remark} 

Remark \ref{summaryrem} and standard arguments give the main theorem of this section.

\begin{thrm}\label{mainthrmdetail} Suppose that $A \subseteq \bbR$ is such that $L(A, \bbR) \models \AD^{+}$, and $G \subseteq \bbP$ is $L(A, \bbR)$-generic. Then
$L(A, \bbR)[G]$ satisfies $\ZFC$ along with the following statements. 
\begin{enumerate}
\item $2^{\aleph_{0}} = \aleph_{2}$
\item $\exists B \subseteq \omega_{2}\, V = L[B]$,
\item $\NS_{\omega_{1}}$ is saturated, 
\item\label{goodseq} $\bigcup \{j_{p,G}(F) : p = (M, F, a) \in G\}$ is a very strongly increasing sequence of length $\omega_{2}$. 
\end{enumerate} 
\end{thrm}

\begin{proof}(Sketch)  We sketch the proof of part (\ref{goodseq}); the other parts follow from standard $\pmax$ arguments. For each condition $p = (M, F, a)$ in the generic filter $G$, 
the elementarity of $j_{p,G}$ gives that $j_{p,G}(F)$ is a very strongly increasing sequence. The order on $\bbP$ implies that the union of these sequences $j_{p,G}(F)$ is a very strongly increasing sequence of length at most $\omega_{2}^{V}$. To see that the length is in fact $\omega_{2}^{V}$, fix $\gamma < \omega_{2}^{V}$, a prewellording $\leq$ of $\omega_{1}^{V}$ of length $\gamma$ and, by part (2) of Remark \ref{summaryrem}, a condition $p=(M,F,a)\in G$ and a prewellordering $\leq_{0}$ of $\omega_{1}^{M}$ in $M$ such that $\leq = j_{p,G}(\leq_{0})$.  Lemma \ref{extendlem} and the genericity of $G$ imply that there exists a condition $p' = (M', F', a') < (M, F, a)$ in $G$ (with the order witnessed by an iteration $j$ of $(M, \NS_{\omega}^{M})$) such that the length of $F'$ is greater than the length of the wellordering $j(\leq_{0})$. The elementarity of $j_{p',G}$ then implies that $j_{p',G}(F')$ has length greater than that of $\leq = j_{p',G}(j(\leq_{0}))$. 
\end{proof} 

\begin{remark}\label{spitrem}
One could naturally try to reproduce the result proved here using an iterated forcing consisting alternately of Hechler forcing and adding club subsets of $\omega_{1}$. One issue with this approach is that if $F= \{ f_{\alpha} : \alpha < \omega_{1} \}\subseteq \breals$ is $\leq^{*}$-unbounded, and $g \in \breals$ is added by a c.c.c. forcing over $V$, then there are cofinally many $m \in \omega$ such that the set $\{ \alpha < \omega_{1} : f_{\alpha}(m) \geq g(m) \}$ is stationary (in $V$). If $F$ is the range of a cofinal subsequence of some strongly increasing sequence (whose index set is closed in the length of the sequence) then the stationarity of some if these sets will have to be destroyed into order to use $g$ to continue the sequence. The argument from Lemma \ref{newlem} avoids this issue by requiring preservation only of the sets $j_{\alpha, \omega_{1}}(A)$ coming from the models in the sequence being built. That is, it is not necessary to preserve the stationarity of every stationary subset of $\omega_{1}$ in $V$. The Semi-Properness Iteration Lemma on pages 485-486 of \cite{PIF} offers a parallel degree of freedom, as (using the notation there) for each successor $j$ the quotient $P_{j}/P_{i}$ is required to be semi-proper only for arbitrarily large nonlimit $i < j$. 
\end{remark}

We conclude with a couple of remaining open questions. We first remark that the 
$\pmax$ machinery is well-suited for constructing models 
with strongly increasing sequences of length $\omega_2$ but does not seem to be readily 
adaptable to construct models with longer strongly increasing sequences. Corollary \ref{no4cor} implies that 
there cannot exist strongly increasing sequences of length $\omega_4$ of functions from 
$\omega$ to $\omega$, but the corresponding question about $\omega_3$ remains open:

\begin{ques}
  Is it consistent with $\mathsf{ZFC}$ that there exists a strongly increasing sequence of 
  length $\omega_3$ consisting of functions from $\omega$ to $\omega$?
\end{ques}

Lastly, we forced over a model satisfying $\mathrm{AD} + V = L(\mathbb{R})$ to obtain 
our main consistency result. We do not know if these hypotheses are optimal, or even if 
the existence of a strongly increasing $\omega_2$-sequence of functions from $\omega$ to 
$\omega$ carries any large cardinal strength at all. The approach suggested by Remark \ref{spitrem} may answer the following question. 

\begin{ques}
  What is the consistency strength of ``$\mathsf{ZFC} + $ there exists a strongly 
  increasing $\omega_2$-sequence of functions from $\omega$ to $\omega$"?
\end{ques}

\bibliographystyle{plain}
\bibliography{CLHqbib}

\noindent Institute of Mathematics, CAS\\
\v{Z}itn\'{a} 25\\
115 67 Praha 1\\
Czechia\\
lambiehanson@math.cas.cz

\vspace{\baselineskip}

\noindent Department of Mathematics\\ Miami University\\ Oxford, Ohio\\ USA\\
larsonpb@miamioh.edu

\end{document}